\documentclass[11pt,reqno]{amsart}
\usepackage{setspace}
\usepackage{amsmath, amssymb, amscd, amsthm, amsfonts}
\usepackage{graphicx}

\usepackage{enumitem}
\usepackage{mathdots}
\usepackage{mathpazo}

\usepackage{indentfirst}
\usepackage{amssymb,amsmath,amsfonts,amsthm}
\usepackage[all]{xy}
\usepackage{enumerate}
\usepackage{mathrsfs}
\usepackage{graphicx}
\numberwithin{equation}{section}
\usepackage[all]{xy}
\usepackage{amssymb,amscd,amsthm,amsmath,graphicx,color}
\usepackage{mathrsfs}
\usepackage[english]{babel}
\usepackage[utf8x]{inputenc}

\newtheorem{theorem}{Theorem}[section]
\newtheorem*{proposition*}{Proposition}
\newtheorem{proposition}[theorem]{Proposition}
\newtheorem{corollary}[theorem]{Corollary}
\newtheorem{lemma}[theorem]{Lemma}
\newtheorem{example}[theorem]{Example}
\newtheorem{remark}[theorem]{Remark}
\newtheorem{definition}[theorem]{Definition}
\newtheorem{question}[theorem]{Question}

\DeclareMathOperator{\Hom}{Hom}
\DeclareMathOperator{\Ext}{Ext}
\DeclareMathOperator{\Tor}{Tor}

\DeclareMathOperator{\Supp}{Supp}
\DeclareMathOperator{\depth}{depth}

\DeclareMathOperator{\pd}{pd}
\DeclareMathOperator{\cd}{cd}
\DeclareMathOperator{\gdim}{G-dim}
\DeclareMathOperator{\id}{id}

\DeclareMathOperator{\Spec}{Spec}

\DeclareMathOperator{\Ass}{Ass}

\DeclareMathOperator{\ann}{ann}

\title{Cohen-Macaulay pairs}

\author[\,R. Holanda]{\,Rafael Holanda}
\address{Departamento de Matemática, Universidade Federal da Paraíba - 58051-900, João Pessoa, PB, Brazil}
\email{rf.holanda@gmail.com}

\author[\,C. B. Miranda-Neto]{\,Cleto B. Miranda-Neto}
\address{Departamento de Matemática, Universidade Federal da Paraíba - 58051-900, João Pessoa, PB, Brazil}
\email{cleto@mat.ufpb.br}

\date{\today}
\thanks{Corresponding author: C. B. Miranda-Neto, email: {\tt cleto@mat.ufpb.br}}

\begin{document}

\begin{abstract} The main purpose of this note is to extend and establish a new approach to the concept of (relative) Cohen-Macaulayness, by investigating the cohomological dimension as well as the depth of a pair of modules over a commutative Noetherian ring. The notion also largely extends the cohomologically complete intersection property of Hellus and Schenzel. As crucial tools, we use Herzog's theory of generalized local cohomology along with spectral sequence techniques. We also provide byproducts concerning two classical problems in homological commutative algebra. \end{abstract}

\maketitle

\section{Preliminaries: Main concepts}

Throughout this note, by a {\it ring} we mean a commutative Noetherian ring with non-zero identity. We let $R$ stand for a ring and $I$ a proper ideal of $R$. By a {\it finite} $R$-module we mean a finitely generated $R$-module. Background on commutative and homological algebra can be found in \cite{BH}, \cite{W}.

The main purpose of this work is to investigate extensions of the classical Cohen-Macaulay property in terms of the (non-)vanishing of generalized local cohomology modules, expressed in terms of appropriate numerical invariants -- precisely, the cohomological dimension and the depth of a pair of modules, which are defined below in this section. In particular, the property of Cohen-Macaulayness for pairs of modules is introduced in this paper (see Definition \ref{pair-def}). Applications concerning two celebrated problems in homological commutative algebra are given (see subsections \ref{Hu} and \ref{AR}), as well as a new characterization of cohomologically complete intersection ideals in the sense of \cite{HS} (see Corollary \ref{nice-appl}). In the sequel, we introduce the main definitions.

\begin{definition}\label{gen-cohomol}{\rm (\cite{H}, also \cite{B})} \rm  Let $M, N$ be $R$-modules. For an integer $i\geq 0$, the {\it $i$-th generalized local cohomology module of $M, N$ with respect to $I$} is defined by
$$H^{i}_I(M, N)   =  \displaystyle \varinjlim_{q}{\rm Ext}_R^{i}(M/I^qM,\,N),$$ where the direct system is induced by the natural maps $M/I^{q+1}M\to M/I^qM$. Notice that, letting $M=R$, this definition retrieves the classical local cohomology module $H^{i}_I(N)$.

\end{definition}

For more on the above notion, see also \cite{FJMS, HZ, S}. In particular, in \cite{FJMS}, various compelling reasons are given for the study of generalized local cohomology, including a characterization of Cohen-Macaulay modules, the development of a more general theory of deficiency and canonical modules, and criteria for prescribed bound on projective dimension of modules over quotients of Gorenstein local rings. In this sense, the present note can be regarded as a continuation of \cite{FJMS}.

Now, recall for completeness that the $I$-depth of a finite $R$-module $N$, denoted ${\rm depth}_IN$, can be defined as the length of a maximal $N$-sequence contained in $I$, or equivalently, the grade of $I$ on $N$ (note this number is zero if $I=(0)$). As usual, in order to extend this definition for general $N$, we set 
$\depth_IN=\inf\{i\geq 0 \mid H^i_I(N)\neq0\}$. Also, the cohomological dimension of $N$ with respect to $I$ is defined as $\cd_IN=\sup\{i\geq 0 \mid H^i_I(N)\neq0\}$. Below we recall definitions in the more general setting of pairs of modules, which plays a major role in this note.

\begin{definition}\label{defs}\rm
Consider a pair of $R$-modules $M, N$.
\begin{itemize}
    \item [(i)] The {\it depth of the pair $M, N$ with respect to $I$} is $$\depth_I(M,N)=\inf\{i\geq 0 \mid H^i_I(M,N)\neq0\};$$
    \item [(ii)] The {\it cohomological dimension of the pair $M, N$ with respect to $I$} is $$\cd_I(M,N)=\sup\{i\geq 0 \mid H^i_I(M,N)\neq0\}.$$
\end{itemize} Note these invariants extend the classical ones by letting $M=R$, i.e., we have $\depth_I(R,N)=\depth_IN$ and $\cd_I(R,N)=\cd_IN$. Also, if $(R, \mathfrak{m})$ is local and $I=\mathfrak{m}$, we write $\depth_{\mathfrak{m}}(M,N)=\depth_R(M,N)$ and $\cd_{\mathfrak{m}}(M,N)=\cd_R(M,N)$.
\end{definition}

\begin{remark}\label{BiZa}\rm In \cite[Proposition 5.5]{B} it was proved that, for any pair of finite $R$-modules $M, N$,
$$\depth_I(M,N)=\depth_{\ann_RM/IM}N.$$ If $R$ is local, this formula retrieves the equality $\depth_R(M,N)=\depth_RN$ obtained in \cite[Theorem 2.3]{S}.
\end{remark}

\begin{definition}\rm (\cite{R}) An $R$-module $N$ is ({\it relative}) {\it Cohen-Macaulay with respect to $I$} if  $\depth_IN=\cd_IN$.

\end{definition}

Note the above concept extends the classical Cohen-Macaulay property; indeed, if $R$ is local and $N$ is a finite Cohen-Macaulay $R$-module, then $\depth_RN= \cd_RN$. Now, in general, it is clear from the definitions that $\depth_IN\leq \cd_IN$, whereas for pairs
$$\depth_I(M,N)\leq \cd_I(M,N).$$ Such facts motivated us to introduce the following concept.

\begin{definition}\label{pair-def}\rm 
We say that a pair of $R$-modules $M,N$ is a {\it Cohen-Macaulay pair with respect to $I$} if $$\depth_I(M,N)=\cd_I(M,N).$$ More precisely, for a non-negative integer $t$, we say that $M, N$ is a {\it $t$-Cohen-Macaulay pair with respect to $I$} if $M, N$ is a Cohen-Macaulay pair with respect to $I$ and $\depth_I(M,N)=t$.
Furthermore, if $(R,\mathfrak{m})$ is local, we simply say that $M,N$ is a {\it Cohen-Macaulay pair} if $M, N$ is a Cohen-Macaulay pair with respect to $\mathfrak{m}$. 
\end{definition}

\begin{remark}\label{cci} \rm It is clear that the pair $R,N$ is Cohen-Macaulay with respect to $I$ if and only if $N$ is Cohen-Macaulay with respect to $I$. Therefore, Definition \ref{pair-def} generalizes the concept of relative Cohen-Macaulayness. By letting in addition $N=R$, the concept also extends the cohomologically complete intersection property of \cite{HS} (the ideal $I$ is said to have such a property if  ${\rm grade}\,I=\cd_IR$).
\end{remark}

\begin{remark}\label{depth-finite}\rm Suppose $M, N$ are finite $R$-modules. Since $\depth_IN\leq \depth_{\ann_RM/IM}N$, we can use Remark \ref{BiZa} to write $\depth_IN\leq \depth_I(M,N)$. Thus it seems natural to ask whether (or when) the quantity $\cd_IN$ fits somewhere in the inequalities $$\depth_IN\leq\depth_I(M,N)\leq\cd_I(M,N)$$ (which later in Theorem \ref{c+f} we will prove to hold for general $N$, among other facts). For instance, is it true that necessarily $\cd_I(M,N)\leq \cd_IN$? The answer is no (see Example \ref{CMnotCMpair} below), and in fact in Proposition \ref{cd=h+t} we shall detect conditions under which $\cd_I(M,N)\geq \cd_IN$ takes place. 
\end{remark}

\begin{example}\label{CMnotCMpair}\rm Consider $R=k[\![x,y]\!]/(xy)$, where $k$ is a field. It can be checked that
$$H^{2i}_{xR}(R/xR,yR)\cong\Ext^{2i}_R(R/xR,yR)\neq 0 \quad \mbox{for\, all} \quad i\geq 0,$$ whereas $H^j_{xR}(yR)=0$ for all $j\neq0$. In other words, we have $$\cd_{xR}(R/xR,yR)=\infty \quad \mbox{and} \quad \cd_{xR}yR=0.$$ This also shows that the pair $R/xR, yR$ is not a Cohen-Macaulay pair with respect to $xR$, while on the other hand $yR$ is Cohen-Macaulay with respect to $xR$. 
We wonder how to find an example (if any) where $\cd_IN <\cd_I(M,N)<\infty$.
\end{example}

Finally, the main technical tools used in our proofs are the following spectral sequences.

\begin{lemma}{\rm(\cite[Proposition 2.1]{FJMS})}\label{ss}
For  $R$-modules $M, N$, with $M$ finite, there are spectral sequences:
\begin{itemize}
    \item [{\rm (i)}] $E_2^{p,q}=H^p_I(\Ext^q_R(M,N))\Rightarrow_p E_{\infty}^{p+q}=H^{p+q}_I(M,N)$;
    \item [{\rm (ii)}] $E_2^{p,q}=\Ext^p_R(M,H^q_I(N))\Rightarrow_p E_{\infty}^{p+q}=H^{p+q}_I(M,N)$.
\end{itemize}
\end{lemma}

\section{Cohomological dimension, depth, and Ext modules}\label{setupsection}

\subsection{Cohomological dimension and Ext modules}

Given two $R$-modules $M,N$, we write $$e_R(M,N)=\sup\{i\geq 0 \mid \Ext^i_R(M,N)\neq0\},$$ which, if $M$ and
$N$ are both finite, is typically dubbed the {\it Gorenstein projective dimension of $M$ relative to $N$}  (see \cite{DH}); in this case, if in addition $R$ is local and $M$ has finite projective dimension, say $\rho$, then by \cite[Lemma 1(iii), p.\,154]{Matsu} we have an equality $e_R(M,N)=\rho$.

We also recall the auxiliary lemma below; it holds quite generally for modules that are only required to be finitely presented (over rings that are not necessarily Noetherian), but the following particular statement suffices for our purposes herein.

\begin{lemma}{\rm(\cite[Proposition 4.7]{CJR})}\label{suppcd}
Let $M, N$ be $R$-modules, with $M$ finite. If $\Supp_RN\subset\Supp_RM$, then $\cd_IN\leq\cd_IM.$
\end{lemma}

Now, in the next proposition,
the inequality as well as the isomorphism
have already been proved in \cite[Theorem 2.5]{DH}. Our proof is different, and we complement the result by characterizing when the formula  $\cd_I(M,N)= \cd_IM\otimes_RN+e_R(M,N)$ holds.

\begin{proposition}\label{m+e}
If $M,N$ are finite $R$-modules, then $\cd_I(M,N)\leq \cd_IM\otimes_RN+e_R(M,N)$. Moreover, if $e_R(M,N)<\infty$, there is an isomorphism $$H^{\cd_IM\otimes_RN+e_R(M,N)}_I(M,N)\cong H^{\cd_IM\otimes_RN}_I(\Ext^{e_R(M,N)}_R(M,N)),$$
and the equality $\cd_I(M,N)= \cd_IM\otimes_RN+e_R(M,N)$ holds if and only if $$\cd_I\Ext^{e_R(M,N)}_R(M,N)=\cd_IM\otimes_RN.$$
\end{proposition}

\begin{proof}
Note $\Supp_R\Ext^{e_R(M,N)}_R(M,N)\subset\Supp_RM\otimes_R N$. Now Lemma \ref{suppcd} gives $$\cd_I\Ext^q_R(M,N)\leq \cd_IM\otimes_RN \quad \mbox{for\, all} \quad q\geq 0,$$ so that the spectral sequence of Lemma \ref{ss}(i) is such that $$E_2^{p,q}=0 \quad \mbox{for\, all} \quad  p>\cd_IM\otimes_RN \quad \mbox{or} \quad q>e_R(M,N).$$ The proposed inequality and isomorphism follow from convergence, and the equivalence is now immediate.
\end{proof}

\smallskip

For a precise description of the relation between $\cd_I(M,N)$ and $\cd_IN$ under appropriate conditions, we introduce the following invariant: $$h_I(M,N)=\sup_{q\geq 0}\{e_R(M,H^q_I(N))\}.$$

\begin{proposition}\label{cd=h+t}
Let $M,N$ be two $R$-modules, with $M$ finite. If $N$ is Cohen-Macaulay with respect to $I$, then $$\cd_I(M,N)=\cd_IN+h_I(M,N).$$ In particular, $\cd_I(M,N)\geq\cd_IN$.
\end{proposition}

\begin{proof}
The spectral sequence of Lemma \ref{ss}(ii) degenerates in such a way that $$\Ext^p_R(M,H^{\cd_IN}_I(N))\cong H^{p+\cd_IN}_I(M,N) \quad \mbox{for\, all} \quad p\geq 0,$$ whence the result.
\end{proof}

\begin{example}\rm It follows from Proposition \ref{cd=h+t} that the invariant $h_I(M,N)$ is not always finite. Indeed, as in Example \ref{CMnotCMpair}, pick a field $k$ and let  $R=k[\![x,y]\!]/(xy)$, $I=xR$, $M=R/xR$, and $N=yR$. We have seen that $N$ is Cohen-Macaulay with respect to $I$ (with $\cd_IN=0$) and that, on the other hand, $\cd_I(M,N)=\infty$. Therefore, the proposition yields $$h_I(M,N)=\infty.$$
\end{example}

Next, we record yet another result, where projective and injective dimensions (of finite modules) over a given local ring $R$ are represented, respectively, by $\pd_R$ and $\id_R$. First, a preparatory lemma is in order.

\begin{lemma}{\rm(\cite[Theorem 3.5]{DH})}\label{cdformula}
Let $R$ be a $d$-dimensional Cohen-Macaulay local ring and let $M, N$ be finite $R$-modules such that either $\pd_RM<\infty$ or $\id_RN<\infty$. Then,
$\cd_R(M,N)= d-\depth_{\ann_RN}M$. In particular, $\cd_R(M,N)\leq d$.
\end{lemma}

\begin{corollary}\label{cd=dim}
Let $(R,\mathfrak{m})$ be a $d$-dimensional Cohen-Macaulay local ring and $M, N$ be finite $R$-modules such that either $\pd_RM<\infty$ or $\id_RN<\infty$. If $N$ is maximal Cohen-Macaulay, then $\cd_R(M,N)=d$ and $$\Ext^i_R(M,H_\mathfrak{m}^d(N))=0 \quad \mbox{for\, all} \quad i>0.$$
\end{corollary}
\begin{proof}
It follows immediately from Proposition \ref{cd=h+t} (with $I=\mathfrak{m}$) and Lemma \ref{cdformula}.
\end{proof}

\subsection{Depth, and more on cohomological dimension versus Ext modules}
Let us first show that if $N$ is not required to be Cohen-Macaulay with respect to $I$ (or even finite), then at least there is an inequality $\cd_I(M,N)\leq\cd_IN+h_I(M,N)$ (cf.\,Proposition \ref{cd=h+t}).
Even more, we have the following statement which also involves depth and gives additional information about cohomological dimension related to the quantity $h_I(M, N)$.

\begin{theorem}\label{c+f}
Let $M,N$ be a pair of $R$-modules, with $M$ finite. The following assertions hold:

\begin{itemize}
    \item [{\rm (i)}] $\depth_IN\leq\depth_I(M,N)\leq\cd_I(M,N)\leq\cd_IN+h_I(M,N)$;
    \item [{\rm (ii)}] There is an isomorphism $$\Ext^{h_I(M,N)}_R(M,H^{\cd_IN}_I(N))\cong H^{\cd_IN+h_I(M,N)}_I(M,N);$$
     \item [{\rm (iii)}] If either $\depth_IN=0$ or $\depth_IN\geq {\rm max}\{1,\, h_I(M,N)\}$, then $$H^{\depth_IN}_I(M,N)\cong\Hom_R(M,H^{\depth_IN}_I(N)).$$
\end{itemize}
\end{theorem}
\begin{proof}
The spectral sequence of Lemma \ref{ss}(ii) is such that $E_2^{p,q}=0$ for all integers $p,q\geq 0$ satisfying $$p>h_I(M,N) \quad \mbox{or} \quad q>\cd_IN \quad \mbox{or} \quad q<\depth_IN.$$ From the convergence of $E$, we obtain $$H^i_I(M,N)=0 \quad \mbox{for\, all} \quad i<\depth_IN \quad \mbox{or} \quad i>\cd_IN+h_I(M,N),$$ from which the inequalities of (i) come. The isomorphism in (ii) follows from the fact that $$E_2^{h_I(M,N),\cd_IN}\cong H^{\cd_IN+h_I(M,N)}_I(M,N).$$ The assertion (iii) follows directly from the convergence of the five-term exact sequence associated to $E$.
\end{proof}

\begin{remark}\label{depthremark}\rm
(i) As already mentioned in Remark \ref{depth-finite}, the inequality $\depth_IN\leq\depth_I(M,N)$ is known to hold when both $M$ and $N$ are finite $R$-modules; so, in particular, part (i) of our theorem generalizes this fact.

\smallskip

(ii) Under the conditions of Theorem \ref{c+f}(iii), we derive that the equality $\depth_IN=\depth_I(M,N)$ is true if and only if the intersection $\Supp_RM\cap\Ass_RH^{\depth_IN}_I(N)$ is non-empty (this clearly holds if, for instance, $\Supp_RM=\Spec R$).

\end{remark}

To close the section, by combining Proposition \ref{m+e} and Theorem \ref{c+f} we immediately obtain the following consequences.

\begin{corollary}\label{bounds}
If $M, N$ is a pair of finite $R$-modules, then
$$\depth_IN\leq\depth_I(M,N)\leq\cd_I(M,N)\leq\min\{\cd_IM\otimes_RN+e_R(M,N),\,\cd_IN+h_I(M,N)\}.$$
\end{corollary}

\begin{corollary}
If $M, N$ are finite $R$-modules such that $\cd_IN+h_I(M,N)=\cd_IM\otimes_RN+e_R(M,N)<\infty$, then $$\Ext^{h_I(M,N)}_R(M,H^{\cd_IN}_I(N))\cong H^{\cd_IN+h_I(M,N)}_I(M,N)\cong H^{\cd_IM\otimes_RN}_I(\Ext^{e_R(M,N)}_R(M,N)).$$
\end{corollary}

\begin{question}\rm Our Corollary \ref{bounds} may suggest that we consider the class of pairs of finite $R$-modules $M, N$ such that 
    $$\depth_I(M,N)=\min\{\cd_IM\otimes_RN+e_R(M,N),\,\cd_IN+h_I(M,N)\},$$ which is thus contained in that formed by the Cohen-Macaulay pairs with respect to $I$ (see Definition \ref{pair-def}). Is such a containment strict?
\end{question}

\subsection{Application: Generalized Huneke's problem}\label{Hu} Given an $R$-module $N$, denote by ${\rm Ass}_RN$ its set of associated primes. The well-known Huneke's conjecture, raised in \cite[Conjecture 5.1]{Hunekeconjecture}, predicts that if $R$ is local and $N$ is finite then, for all $i \geq 0$, the set ${\rm Ass}_RH^i_{I}(N)$ is finite. There exist counterexamples to this conjecture -- the first one was given in \cite{Ka} -- but the issue as to when  ${\rm Ass}_RH^i_{I}(N)$ is finite at least for some $i$ (it is well-known that $i\geq 2$ is the situation of interest) remains intriguing; see \cite{BBLSZ} and its references on the theme. Here we consider the following natural generalization of this problem.

\medskip

\noindent {\it Generalized Huneke's question}: If $M, N$ are finite, when is the set $\Ass_RH^i_{I}(M, N)$ finite for some $i\geq 2$? 

\medskip

In fact we deal here with the more general situation where $N$ is only assumed to be weakly Laskerian, in the following sense.

\begin{definition}\rm An $R$-module $N$ is said to be \textit{weakly Laskerian} if the set $\Ass_RN/U$ is finite for every $R$-submodule $U$ of $N$.
    
\end{definition}

Note for completeness that any Noetherian module (e.g., a finite $R$-module) is weakly
Laskerian. If $\Supp_RN$ is finite, then $N$ is weakly
Laskerian; in particular, any Artinian $R$-module is weakly Laskerian. If $M$ is finite and $N$ is weakly Laskerian, then the $R$-modules $\Ext_R^i(M, N)$ and $\Tor^R_i(M, N)$ are weakly Laskerian for all $i \geq 0$. We refer to \cite{DAM} for details. For instance, the
following auxiliary lemma is a special case of \cite[Theorem 3.1 and Remark 2.7(ii)]{DAM}.

\begin{lemma}\label{ord-finite} Let $N$ be a weakly Laskerian $R$-module. Suppose there exists an integer $c\geq 0$ such that $H_I^i(N)=0$ for all $i\neq c$. Then, $\Ass_RH^c_{I}(N)$ is finite.
\end{lemma}

Our contribution to the generalized Huneke's question is the following byproduct of Theorem \ref{c+f}.

\begin{corollary} Let $N$ be weakly Laskerian and $M$ be finite. Assume that $N$ is Cohen-Macaulay with respect to $I$, and that either $$\cd_IN=0 \quad \mbox{or} \quad \cd_IN\geq {\rm max}\{1,\, h_I(M,N)\}.$$ Then, the set ${\rm Ass}_RH^{\cd_IN}_I(M,N)$ is finite.
\end{corollary}
\begin{proof} Let $c=\cd_IN$ and note $\depth_IN=c$ by hypothesis, so that $$H_I^i(N)=0 \quad \mbox{for\, all} \quad i\neq c.$$ By Lemma \ref{ord-finite}, the set $\Ass_RH^c_{I}(N)$ is finite. On the other hand, by Theorem \ref{c+f}(iii), we derive that
$${\rm Ass}_RH^c_I(M,N)=\Supp_RM\cap\Ass_RH^c_I(N),$$ which therefore must also be a finite set, as needed.
\end{proof}

\section{Cohen-Macaulayness}\label{CMsection}

\subsection{Modules versus pairs} Here we apply, in particular, results from the previous section to investigate the interplay between (relative) Cohen-Macaulay modules and our concept of Cohen-Macaulay pairs (see Definition \ref{pair-def}).

The first result is a direct consequence of Theorem \ref{c+f}(i).

\begin{corollary}\label{c+fcor}
Let $M, N$ be a pair of $R$-modules, with $M$ finite. The following assertions hold:
\begin{itemize}
\item [{\rm (i)}] If $N$ is Cohen-Macaulay with respect to $I$ and $h_I(M,N)=0$, then $M,N$ is a Cohen-Macaulay pair with respect to $I$;
\item [{\rm (ii)}] Assume that $M,N$ is a $(\cd_IN+h_I(M,N))$-Cohen-Macaulay pair with respect to $I$ and that the equality $\depth_IN=\depth_I(M,N)$ holds {\rm (}e.g., if ${\ann_RM/IM}=I$; see Remark \ref{BiZa}{\rm )}. Then, $N$ is Cohen-Macaulay with respect to $I$ if and only if $h_I(M,N)=0$.
\end{itemize}
\end{corollary}

Next we are able to characterize the Cohen-Macaulayness of a pair (with respect to $I$), provided that it admits a single non-zero Ext module. The result will also serve as a crucial tool later on, in Subsection \ref{AR}.

\begin{proposition}\label{cohomologicalformulas}
If $M, N$ is a pair of $R$-modules, with $M$ finite, such that there is a  non-negative integer $e$ with $\Ext^i_R(M,N)=0$ for all $i\neq e$, then $$\depth_I(M,N)=\depth_I\Ext^e_R(M,N)+e \quad \mbox{and} \quad \cd_I(M,N)=\cd_I\Ext^e_R(M,N)+e.$$
In this case, $\Ext^e_R(M,N)$ is Cohen-Macaulay with respect to $I$ if and only if $M,N$ is a Cohen-Macaulay pair with respect to $I$.
\end{proposition}
\begin{proof}
Using the hypothesis, the spectral sequence of Lemma \ref{ss}(i) degenerates in such a way that $$H^{p+e}_I(M,N)\cong H^p_I(\Ext^e_R(M,N)) \quad \mbox{for\, all} \quad p\geq 0,$$ whence the result.
\end{proof}

\begin{remark}\label{local} \rm If $(R, \mathfrak{m})$ is local and $M, N$ are both finite, then as mentioned in Remark \ref{BiZa} we have the equality $\depth_R(M,N)=\depth_RN$ obtained in \cite[Theorem 2.3]{S}.
Along with Proposition \ref{cohomologicalformulas} (let $I=\mathfrak{m}$ there), this gives that if $e$ is a non-negative integer such that $\Ext^i_R(M,N)=0$ for all $i\neq e$, then $$\depth_RN=\depth_R\Ext^e_R(M,N)+e\quad\mbox{and}\quad\cd_R(M,N)=\dim_R\Ext^e_R(M,N)+e.$$
\end{remark}

The characterization below follows easily by Lemma  \ref{cdformula} (and again \cite[Theorem 2.3]{S}). 

\begin{corollary} Assume the conditions of Lemma  \ref{cdformula}. The following assertions are equivalent:
\begin{itemize}
    \item [{\rm (i)}] $N$ is maximal Cohen-Macaulay;
    \item [{\rm (ii)}] $M,N$ is a $d$-Cohen-Macaulay pair;
    \item [{\rm (iii)}] $M,N$ is a $d$-Cohen-Macaulay pair and $\depth_{\ann_RN}M=0$.
\end{itemize}
\end{corollary}

\subsection{Semidualizing modules} In this part we  study the role of semidualizing modules in our investigation (for details about such modules we refer to \cite{J-L-Sean} and its suggested references). We will prove, for instance, that if $C$ is semidualizing then $R$ is Cohen-Macaulay with respect to $I$ if and only if $C,C$ is a Cohen-Macaulay pair with respect to $I$ (see Corollary \ref{nice-appl}).

\begin{definition}\rm
A finite $R$-module $C$ is a {\it semidualizing} module if the following conditions hold:
\begin{itemize}
    \item [(i)] The homothety homomorphism $R\rightarrow\Hom_R(C,C)$ is an isomorphism;
    \item [(ii)] $e_R(C,C)=0$.
\end{itemize}
\end{definition}

\begin{example}\rm First, $R$ is semidualizing as a module over itself. If $R$ is a Cohen-Macaulay local ring having a canonical module $\omega_R$, then $\omega_R$ is dualizing, hence semidualizing. While, for instance, every semidualizing module over a Gorenstein local ring is necessarily free, there are plenty of examples of Cohen-Macaulay local rings $R$ admitting a semidualizing module $C$ which is neither free nor dualizing. One of them can be described as follows. Let $k$ be a field and consider the local epimorphism 
$$S = k[\![x, y, t, u]\!]/(x^2, xy, y^2)\rightarrow R  =  k[\![x, y, t, u]\!]/(x^2, xy, y^2, t^2, tu, u^2).$$ Then, according to \cite[Example 1.2]{J-L-Sean}, the $R$-module
$C  =  {\rm Ext}_S^2(R, S)$ is semidualizing but neither free nor dualizing.
\end{example}

Throughout this subsection, $C$ stands for a semidualizing $R$-module. For the next definition, given an $R$-module $M$, we write $M^C=\Hom_R(M,C)$.

\begin{definition}\label{totref}\rm
A finite $R$-module $M$ is {\it totally $C$-reflexive} if the following conditions hold:
\begin{itemize}
    \item [(i)] $M$ is {\it $C$-reflexive}, i.e., the canonical homomorphism $M\rightarrow M^{CC}$ is an isomorphism; 
    \item [(ii)] $e_R(M,C)=e_R(M^C,C)=0$.
\end{itemize}
\end{definition}

Now, by taking $N=C$ and $e=0$ in Proposition \ref{cohomologicalformulas}, we derive the following fact.

\begin{corollary}\label{cohomologicalformulaswithR}
If $M$ is a finite $R$-module such that $e_R(M,C)=0$ {\rm (}e.g., if $M$ is totally $C$-reflexive{\rm )}, then $$\depth_I(M,C)=\depth_IM^C\quad\mbox{and}\quad\cd_I(M,C)=\cd_IM^C.$$ In this case, $M^C$ is Cohen-Macaulay with respect to $I$ if and only if $M,C$ is a Cohen-Macaulay pair with respect to $I$.
\end{corollary}

We immediately record the case $M=C$ of Corollary \ref{cohomologicalformulaswithR}.

\begin{corollary}\label{nice-appl} We have
$\depth_I(C,C)=\depth_IR$ and $\cd_I(C,C)=\cd_IR$.
In particular, $R$ is Cohen-Macaulay with respect to $I$ $($i.e., $I$ is cohomologically a complete intersection; see Remark \ref{cci}$)$ if and only if $C,C$ is a Cohen-Macaulay pair with respect to $I$.
\end{corollary}

\begin{corollary}
Suppose $R$ is local and $M$ is totally $C$-reflexive. Then, $$\cd_R(M,C)= \cd_R(M^C,C)=\dim_RM.$$
In particular, $\cd_R(C,C)=\dim R$.
\end{corollary}
\begin{proof} First, we apply Corollary \ref{cohomologicalformulaswithR} with $I=\mathfrak{m}$ and use that $M$ is $C$-reflexive, to obtain $$\cd_R(M,C)=\dim_RM^C \quad \mbox{and} \quad \cd_R(M^C,C)=\dim_RM.$$ On the other hand, again because $M^{CC}\cong M$, the $R$-modules $M$ and $M^C$ have the same annihilator, hence the same dimension. Finally, noticing that $R$ is totally $C$-reflexive and letting $M=R$ in the formula $\cd_R(M^C,C)=\dim_RM$, we get the particular assertion.\end{proof}

\subsection{Application: Auslander-Reiten conjecture}\label{AR} We close the paper by applying our Proposition \ref{cohomologicalformulas} to tackle a  case of a famous problem in homological algebra. More precisely, we consider the following (commutative) local version of the conjecture from \cite[p.\,70]{AR}.

\medskip

\noindent {\it Auslander-Reiten conjecture}: Suppose $R$ is local. If $M$ is a finite $R$-module such that $$e_R(M, R)=e_R(M, M)=0,$$ then $M$ is a free $R$-module.

\medskip

This problem has been extensively explored in the literature but remains open if, e.g., $R$ is Gorenstein. We refer to \cite{H-MN} and the references suggested there. One of the key tools for the two corollaries below is the following lemma, which is a special case of \cite[Theorem 3.14]{H-MN}. We adopt the standard notation $M^*={\rm Hom}_R(M, R)$.

\begin{lemma}\label{mainlemma}
Let $R$ be a Cohen-Macaulay local ring of dimension $d\geq2$, and let $M$ be a finite $R$-module. Assume $n$ is an integer such that $1\leq n\leq d-1$. Then, $M$ is a free $R$-module if the following conditions hold:  \begin{itemize}
    \item [{\rm (i)}] $M$ is locally of finite projective dimension in codimension $n$;
    \item [{\rm (ii)}] $M$ is a reflexive $R$-module;
    \item [{\rm (iii)}] $\Ext^i_R(M^*, R)=0$ for all $i=1, \ldots, d$;
    \item [{\rm (iv)}] $\Ext^j_R(M, M)=0$ for all $j=n, \ldots, d-1$.
\end{itemize}
\end{lemma}

First we record yet another freeness characterization for modules over a Cohen-Macaulay local ring $R$. We consider, in particular, the notion of Gorenstein dimension $\gdim_RN$ of finite $R$-modules $N$; see \cite{AuB} for the classical theory. We also observe that, since  $\gdim_RN$ is always finite if $R$ is Gorenstein, our result largely extends \cite[Corollary 10]{A} (which is proved with the aid of stable cohomology modules).

\begin{corollary} Let $R$ be a Cohen-Macaulay local ring of dimension $d\geq2$, and let $M$ be a finite $R$-module. Then, $M$ is a free $R$-module if the following conditions hold:  \begin{itemize}
    \item [{\rm (i)}] $M$ is locally of finite projective dimension on the punctured spectrum of $R$ {\rm (}e.g., if $R$ is an isolated singularity or if $M$ defines a vector bundle{\rm )};
    \item [{\rm (ii)}] $M$ is a reflexive $R$-module;
    \item [{\rm (iii)}] $e_R(M, R)=0$;
    \item [{\rm (iv)}] $\Ext^{d-1}_R(M, M)=0$;
    \item [{\rm (v)}] $\gdim_RM^*<\infty$.
    \end{itemize}
\end{corollary}
\begin{proof} First, because $e_R(M, R)=0$, we can apply Proposition \ref{cohomologicalformulas} (along with Remark \ref{local}) with $N=R$ and $e=0$ to conclude that $M^*$ is maximal Cohen-Macaulay. Now, as $\gdim_RM^*<\infty$, the classical Auslander-Bridger formula gives $$\gdim_RM^*=d-{\rm depth}_RM^*=0,$$ which is equivalent to $M^*$ being totally reflexive (see Definition \ref{totref} with $C=R$); in particular, $$\Ext^i_R(M^*, R)=0 \quad \mbox{for\, all} \quad i=1, \ldots, d.$$ We now apply Lemma \ref{mainlemma} with $n=d-1$. \end{proof}

\smallskip

Finally, we provide a result -- some variant of which is possibly known to the experts -- toward the Auslander-Reiten conjecture. The proof we have found is rather simple and makes use of generalized local cohomology as a new tool to tackle the problem, and also as a new alternative replacement for the current, often hard, derived category methods that have been used in the literature.

\begin{corollary}\label{appAR} Let $R$ be a Cohen-Macaulay local ring of dimension $d\geq2$. Then, the Auslander-Reiten conjecture is true if the following three conditions hold:
    \begin{itemize}
    \item [{\rm (i)}] $M$ is locally of finite projective dimension {\rm (}e.g., locally free{\rm )} in codimension $1$;
    \item [{\rm (ii)}] $M$ is a reflexive $R$-module;
    \item [{\rm (iii)}] $\gdim_RM^*<\infty$.
\end{itemize}
\end{corollary}
\begin{proof} Since again $e_R(M, R)=0$ and $\gdim_RM^*<\infty$, we can use the argument of the preceding proof to obtain $$e_R(M^*, R)=0.$$ Now the freeness of $M$ follows readily by Lemma \ref{mainlemma} with $n=1$.
\end{proof}

\smallskip

Recall that maximal Cohen-Macaulay modules over Gorenstein local rings are reflexive. Now the Gorenstein case of Corollary \ref{appAR} is immediately seen to be the following (cf.\,also \cite[Corollary 4]{A}).

\begin{corollary} Let $R$ be a Gorenstein local ring of dimension $d\geq2$. Then, the Auslander-Reiten conjecture is true if $M$ is locally of finite projective dimension in codimension $1$.
\end{corollary}

\begin{remark}\rm As seen above, the condition $e_R(M, R)=0$ forces $M^*$ to have maximal depth (in the present situation where $R$ is a Cohen-Macaulay local ring). Note it is also the case that ${\rm cd}_R(M, R)=\dim_RM^*$. Now, concerning the depth of $M$ itself, what we derive by assuming $e_R(M, M)=0$ (and applying Proposition \ref{cohomologicalformulas} with $N=M$) is that $${\rm depth}_RM={\rm depth}_R{\rm End}_R(M),$$ where ${\rm End}_R(M)$ stands for the endomorphism module of $M$. In addition, we obtain $${\rm cd}_R(M, M)=\dim_R{\rm End}_R(M).$$
It would therefore be interesting to investigate whether such facts shed any light on the Auslander-Reiten conjecture. 
\end{remark}

\bigskip

\noindent{\bf Acknowledgements.} The first-named author thanks IMPA for the Summer Postdoc 2023 and was partially supported by CNPq, grant 200863/2022-3. The second-named author was partially supported by CNPq, grants 301029/2019-9 and 406377/2021-9.

\end{document}